\documentclass{article}

\usepackage{amsmath,amsthm,amssymb}
\usepackage{graphicx}
\usepackage[numbers]{natbib}

\newtheorem{theorem}{Theorem}[section]
\newtheorem{lemma}[theorem]{Lemma}
\newtheorem{theo}[theorem]{Theorem}
\newtheorem{prop}[theorem]{Proposition}
\newtheorem{coro}[theorem]{Corollary}

\theoremstyle{definition}
\newtheorem{example}[theorem]{Example}
\newtheorem{defn}[theorem]{Definition}
\newtheorem{remark}[theorem]{Remark}

\newcommand{\myaddress}
{\parbox{3in}{\footnotesize
\begin{center}
				Department of Mathematics, City University London,\\
				Northampton Square, London, EC1V 0HB.\\
				{\it E-mail address}: {\tt Elizabeth.Banjo.1@city.ac.uk}
\end{center}
}}

\title{The generic representation theory of the Juyumaya algebra of braids and ties}
\author{ELIZABETH O. BANJO\thanks{This work was carried out while the author was a visiting reseacher in the International Research Visitors Centre of the University of Leeds}\\
				\myaddress}
				
\date{}				
				
\begin{document}	
		
\maketitle
				
\begin{abstract}
In this paper we determine the complex generic representation theory of the Juyumaya algebra. We show that a certain specialization of this algebra is isomorphic to the small ramified partition algebra, introduced by P. Martin. 
\end{abstract}

%
\section{Introduction}
%

The main result of this paper is a determination of the complex generic representation theory of a family of finite dimensional algebras $\{\mathcal{E}_n(x) \colon n \in \mathbb{N}, \; x \in \mathbb{C}\}.$ These algebras were introduced by Juyumaya in \cite{juyumaya99} and studied further by Aicardi and Juyumaya \cite{aicardi} and by Ryom-Hansen \cite{ryom2010}.

The Juyumaya algebras $\mathcal{E}_n(x)$ are a generalisation of the Iwahori-Hecke algebras \cite{mathas}. The complex generic representation theory of the Iwahori-Hecke algebras is reasonably well known (see e.g. \cite{mathas} for a review). Like the Iwahori-Hecke algebras it turns out, as we shall show, that the Juyumaya algebras are generically semisimple. In contrast to the Iwahori-Hecke case however, the generic representation theory of the Juyumaya algebras over the field of complex numbers was only known for the cases $n = 1,2,3$ \cite{aicardi}, \cite{ryom2010}. Here we determine the result for all $n.$   

Our method is to establish, for each $n,$ an isomorphism of $\mathcal{E}_n(1)$ (over $\mathbb{C}$) with an algebra, the small ramified partition algebra $P_n^{\ltimes}$ \cite{martin}, of known complex representation theory and then use general arguments of Cline, Parshall and Scott \cite{cline1999}. 

The paper is organized as follows: we start (in section \ref{ramPart}) by reviewing the small ramified partition algebra $P_n^{\ltimes}$. In section \ref{braidTie}, we recall the definition of the algebras $\mathcal{E}_n(x)$. In section \ref{mainRslt}, we prove that, for each $n$, the algebra $\mathcal{E}_n(1)$ and $P_n^{\ltimes}$ are isomorphic as $\mathbb{C}$-algebras. We use this result as well as other results including arguments in \cite{cline1999} to show that $\mathcal{E}_n(x)$ is semisimple (of given structure) over $\mathbb{C}$ for generic choices of $x$ in section \ref{repthry}.

%
\section{The Small Ramified Partition Algebras} \label{ramPart}
%

In order to define the \emph{small} ramified partition algebra, it will be helpful to recall the definition of the ramified partition algebra, given in \cite{martin2004}, from which this algebra can be constructed. We assume familiarity with the \emph{ordinary} partition algebra \cite{martin1994}.

%
\subsection{Some definitions and notation}
%

For $n \in \mathbb{N},$ we define $\underline{n} = \{1,2, \ldots,n\}$ and $\underline{n}' = \{1',2', \ldots, n'\}.$ Let $S_n$ denote the symmetric group on $\underline{n}$ and $\sigma_{i,i+1}$ the transposition $(i,i+1) \in S_n.$  When we write $\underline{d}$ for a \emph{poset}, we mean $\{1,2, \ldots, d\}$ equipped with the natural partial order $\leq$ (although we will often concentrate on the case $\underline{2} = (\{1,2\}, \; 1 \leq 2)$ in this paper). For a set $X$, we write $\mathcal{P}_X$ for the set of partitions of $X.$ 

\noindent For example,
\begin{eqnarray*}
\mathcal{P}_{\underline{2} \cup \underline{2}'} &=& \{\{\{1\},\{2\},\{1'\},\{2'\}\}, \{\{1,2,1',2'\}\}, \{\{1,2,1'\},\{2'\}\},\\ 
& & \{\{1,2,2'\},\{1'\}\}, \{\{1,1',2'\},\{2\}\}, \{\{2,1',2'\},\{1\}\}, \{\{1,2\},\{1',2'\}\},\\
& & \{\{1,1'\},\{2,2'\}\}, \{\{1,2'\},\{1',2\}\}, \{\{1,2\},\{1'\},\{2'\}\},\\ 
& & \{\{1,1'\},\{2\},\{2'\}\}, \{\{1,2'\},\{1'\},\{2\}\}, \{\{1',2\},\{1\},\{2'\}\},\\ 
& & \{\{2,2'\},\{1\},\{1'\}\}, \{\{1',2'\},\{1\},\{2\}\}\}. 
\end{eqnarray*}

In an element of $\mathcal{P}_{\underline{n} \cup \underline{n}'}$ we call the individual subsets of the set of objects \emph{parts}. For instance, $\{1,2\}$ is a part of the partition $\{\{1,2\},\{1'\},\{2'\}\}$ in $\mathcal{P}_{\underline{2} \cup \underline{2}'}$. For $X' \subset X$ and $c \in P_X$ we define $c|_{X'}$ as the collection of the sets of the form $c_i \cap X'$ where $c_i$ are elements of the partition $c$.

\begin{defn}
For a set $X$, we define the \emph{refinement} partial order on $\mathcal{P}_X$ as follows. For $p, \;q \in \mathcal{P}_X,$ we say $p$ is a refinement of $q,$ denoted $p \leq q,$ if each part of $q$ is a union of one or more parts of $p.$   
\end{defn}

\begin{prop}[See {\citep[Prop. 1]{martin2004}}]
Let $p, \; q \in \mathcal{P}_X$ and $Y \subseteq X.$ Then $p \leq q$ implies $p|_Y \leq q|_Y.$  \qed 
\end{prop}

\begin{remark}
For $F$ a field, $n \in \mathbb{N}, \; \delta' \in F,$ the set $\mathcal{P}_{\underline{n} \cup \underline{n}'}$ is a basis for the partition algebra \cite{martin1994} which we denote by $P_n(\delta').$ The dimension of $P_n(\delta')$ is therefore the Bell number $B_{2n}$ (see \cite{cameron}). The group algebra $FS_n$ of the symmetric group $S_n$ is embedded in $P_n(\delta')$ as the span of the partitions with every part having exactly two elements, one primed and the other unprimed, of $\underline{n} \cup \underline{n}'.$
\end{remark}

%
\subsection{Ramified partition algebra}
%

The ramified partition algebra was introduced by Martin \cite{martin2004} as a generalisation of the ordinary partition algebra $P_n(\delta').$

\begin{defn}
Let $(T, \leq)$ be a finite poset. For a set $X,$ we define $\mathbf{P}_X^{T}$ to be the subset of the Cartesian product $\prod_T\mathcal{P}_X$ consisting of those elements $q = (q_i \colon i \in T)$ such that $q_i \leq q_j$ whenever $i \leq j.$ Any such element $q \in \mathbf{P}_X^{T}$ will be referred to as a \emph{T-ramified partition}.  
\end{defn}

\noindent For example, some elements of $\mathbf{P}_{\underline{2} \cup \underline{2}'}^{\underline{2}}$ are listed below:
 \begin{eqnarray*}
\pi_1 &=& (\{\{1\},\{2\},\{1'\},\{2'\}\}, \{\{1\},\{2\},\{1'\},\{2'\}\}) \\
\pi_2 &=& (\{\{1,2\},\{1'\},\{2'\}\}, \{\{1,2,1'\},\{2'\}\}) \\
\pi_3 &=& (\{\{1,2'\},\{2,1'\}\}, \{\{1,2,1',2'\}\}),
\end{eqnarray*}
and so on.

We now recall from \cite{martin2004} the diagrammatic realization of an element of $\mathbf{P}_{\underline{n} \cup \underline{n}'}^{T}.$ We shall only need the case $T = \underline{2}$ in this paper. We first look at an example therein. The diagram

\[ 
\includegraphics[scale=0.4]{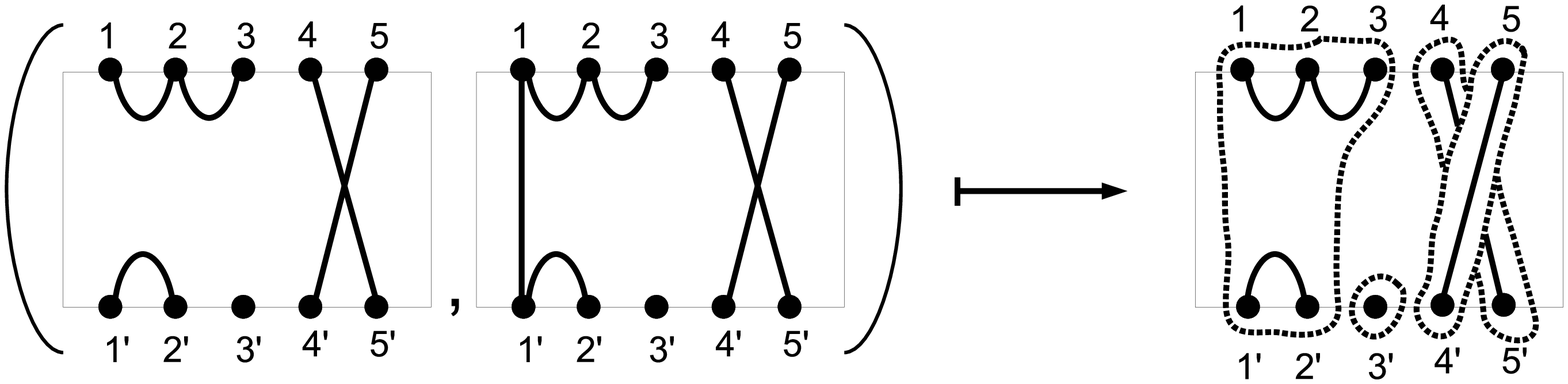}
\] 

\noindent represents $(\{\{1,2,3\},\{1',2'\},\{3'\},\{4,5'\},\{5,4'\}\},\{\{1,2,3,1',2'\},\{3'\},\{4,5'\},$\\
$\{5,4'\}\}).$ For $T = \underline{1}$ case, the diagram coincides with a partition algebra diagram (cf. \cite{halverson2005}, \cite{martin1996}, or \cite{martin2004}). Thus, a diagram of a $\underline{2}$-ramified partition can be thought of as an \emph{enhanced} partition algebra diagram in which every connected component lies inside an ``island'', and the union of components in the island is the less refined part. Note that islands can cross (as illustrated in the diagram above), but it is not hard to draw them unambiguously.

A diagram representing a T-ramified partition is not unique. We say two diagrams are equivalent if they give rise to the same T-ramified partition.

The term \emph{ramified partition diagram} (or sometimes \emph{ramified $2n$-partition diagram} to indicate the number of vertices) will be used to mean the equivalence class of the given diagram.

We refer to the interior line-segments in the underlying partition algebra diagram of a ramified partition diagram as \emph{bones}. 
 
The composition of ramified $2n$-partition diagrams is as follows. First identify the bottom of one ramified $2n$-partition diagram with the top of the other. Then replace bone (resp. island) connected components that are isolated from the boundaries in composition by a factor $\delta_1$ (resp. $\delta_2$) as shown in Figure \ref{compRam}. 

\begin{figure}
\[
\begin{tabular}[c]{@{}c@{}}
\includegraphics[scale=0.4]{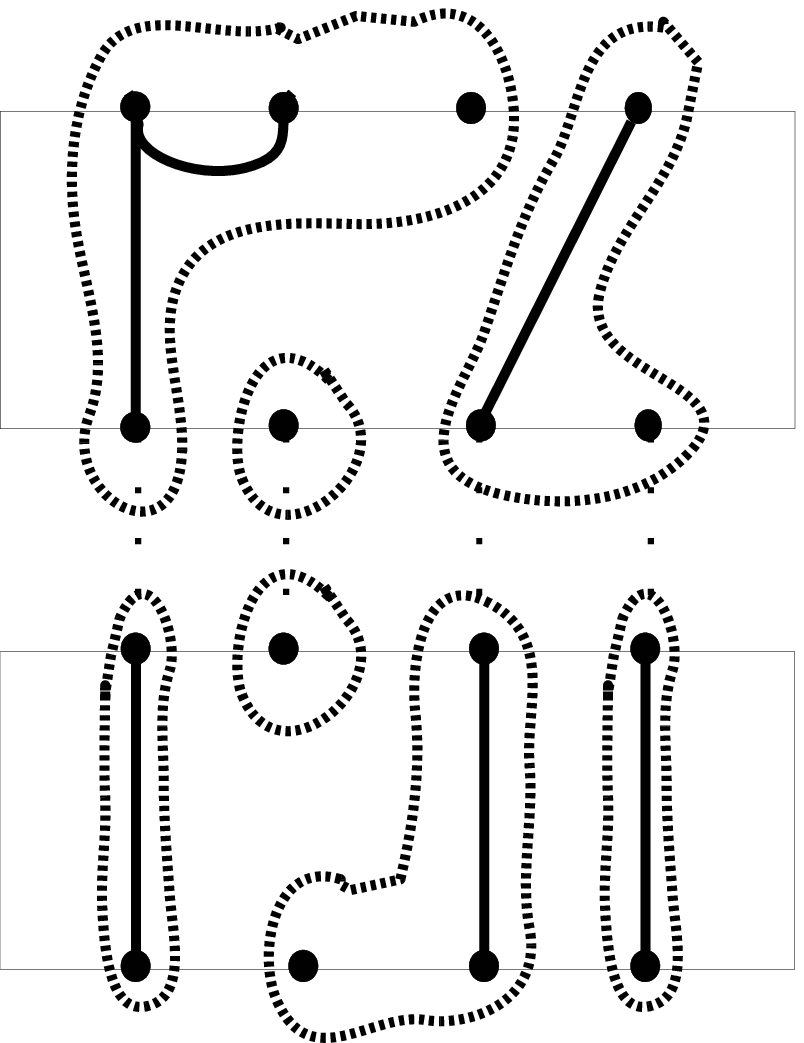}
\end{tabular} \quad = \; \delta_1 \delta_2 \quad
\begin{tabular}[c]{@{}c@{}}
\includegraphics[scale=0.4]{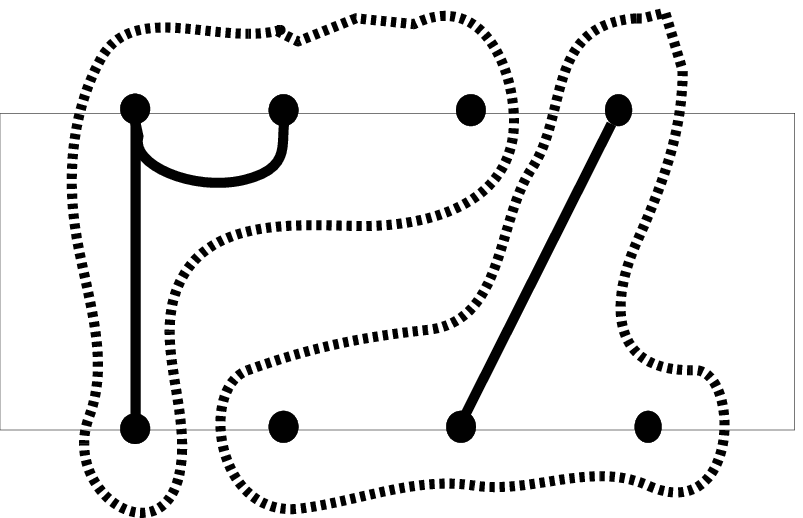}
\end{tabular}
\]
\caption{The composition of diagrams in $P_4^{(\underline{2})}(\delta)$}
\label{compRam}
\end{figure}

Throughout this paper, we shall identify a ramified partition with its ramified partition diagram and speak of them interchangeably.

\begin{prop}[See {\citep[Prop. 2]{martin2004}}]
For any $d$-tuple $\delta = (\delta_1, \ldots, \delta_d) \in F^d,$ the set $\mathbf{P}_{\underline{n} \cup \underline{n}'}^{T}$ forms a basis for a subalgebra of $\bigotimes_{t \in T}P_n(\delta_t).$    
\end{prop}

\begin{proof}
The proof can be found in \cite{martin2004}, but is in essence the observation that bones continue to lie within islands under composition.
\end{proof}

For $\delta = (\delta_1, \delta_2, \ldots, \delta_d)$ $\in F^d,$ we define the \emph{T-ramified partition algebra $P_n^{(T)}(\delta)$} over $F$ as the finite dimensional algebra with basis $\mathbf{P}_{\underline{n} \cup \underline{n}'}^{T}$ and the above composition.

A bone connecting a vertex in the northern boundary to a vertex in the southern boundary of the frame will be called a \emph{propagating line}. 

\noindent The complex generic representation theory of $P_n^{(T)}(\delta)$ has been determined in the case $T = \underline{2}$ in \cite{martin2004}. It was shown that there are many choices of $\delta$ such that $P_n^{(\underline{2})}(\delta)$ is not semisimple for sufficiently large $n,$ but that it is generically semisimple for all $n.$

%
\subsection{Small Ramified Partition Algebra $P_n^{\ltimes}$} \label{sRamPart}
%

In this section we recall the definition of the small ramified partition algebra. To define this algebra we require the following definitions.  

\medskip

\begin{defn}
We define $\text{diag}-\mathcal{P}_n$ to be the subset of $\mathcal{P}_{\underline{n} \cup \underline{n}'}$ such that $i,i'$ are in the same part for all $i \in \mathbb{N}.$ 
\end{defn}

For example, recall from \cite{martin1994} the special elements in $\mathcal{P}_{\underline{n} \cup \underline{n}'}$ as follows.
\begin{eqnarray*}
1 &=& \{\{1,1'\},\; \{2,2'\}, \; \ldots \{i,i'\}, \; \ldots \{n,n'\}\}\\ 
A^{i,j} &=& \{\{1,1'\},\; \{2,2'\}, \; \ldots \{i,i',j,j'\}, \; \ldots \{n,n'\}\}\\
\sigma_{i,j} &=& \{\{1,1'\},\; \{2,2'\}, \; \ldots \{i,j'\},\{j, i'\} \; \ldots \{n,n'\}\} \\
e_i &=& \{\{1,1'\},\; \{2,2'\}, \; \ldots \{i\},\{i'\}, \; \ldots \{n,n'\}\}.
\end{eqnarray*}

Here, $1$ and $A^{i,j}$ are in $\text{diag}-\mathcal{P}_n$.   

\begin{defn}\label{Delta}
For any $\delta' \in F,$ we define $\Delta_n$ as the subalgebra of $P_n(\delta')$ generated by $$\langle S_n, A^{i,j}  (i,j = 1,2, \ldots, n) \rangle.$$ 
\end{defn}

\begin{prop} \label{prop:injMap}
The map $$\ltimes \colon S_n \times \text{diag}-\mathcal{P}_n \to S_n \times \mathcal{P}_{\underline{n} \cup \underline{n}'}$$ given by $$(a,b) \mapsto (a,ba)$$ defines an injective map. 
\end{prop}
\begin{proof}
The well-definedness of $\ltimes$ is clear. To prove that $\ltimes$ is an injective map, it suffices to show that if $(a,ba)$ is equal to $(c,dc)$ in $S_n \times \mathcal{P}_{\underline{n} \cup \underline{n}'}$, then $(a,b)$ is equal to $(c,d)$ in $S_n \times \text{diag}-\mathcal{P}_n.$
Assume that $(a,ba) = (c,dc).$ Since $a = c,$ then $bc = dc.$ But $c$ is invertible, thus, $b=d.$
\end{proof}

\noindent Note that $\ltimes$ is not a surjective map.

\begin{defn}
We define $\mathbb{P}_{\underline{n} \cup \underline{n}'}$ to be the subset of the Cartesian product $S_n \times \mathcal{P}_{\underline{n} \cup \underline{n}'}$ given by the elements $q = (q_1,q_2)$ such that $q_1$ is a refinement of $q_2$.
\end{defn}

\begin{coro}
The set of $\ltimes(S_n \times \text{diag}-\mathcal{P}_n)$ lies in $\mathbb{P}_{\underline{n} \cup \underline{n}'}$ and forms a basis for a subalgebra of $FS_n \otimes_F \Delta_n.$ \qed    
\end{coro}

\begin{defn}
The associative algebra $P_n^{\ltimes}$ over $F$ is the free $F$-module with the set of $\ltimes(S_n \times \text{diag}-\mathcal{P}_n)$ as basis and multiplication inherited from the multiplication on $P_n^{(\underline{2})}(\delta).$ This is the \emph{small ramified partition algebra} \cite{martin}. 
\end{defn}

\noindent It is easy to check that

\begin{lemma}
The multiplication on $P_n^{\ltimes}$ is well-defined up to equivalence. \qed
\end{lemma}

There is a diagram representation of the set of $\ltimes(S_n \times \text{diag}-\mathcal{P}_n)$ since its elements are $\underline{2}$-ramified partitions (see \cite{martin}). 

\begin{example} \label{explRamified}
The map defined in Proposition \ref{prop:injMap} is illustrated by the following pictures.

\includegraphics[scale=0.4]{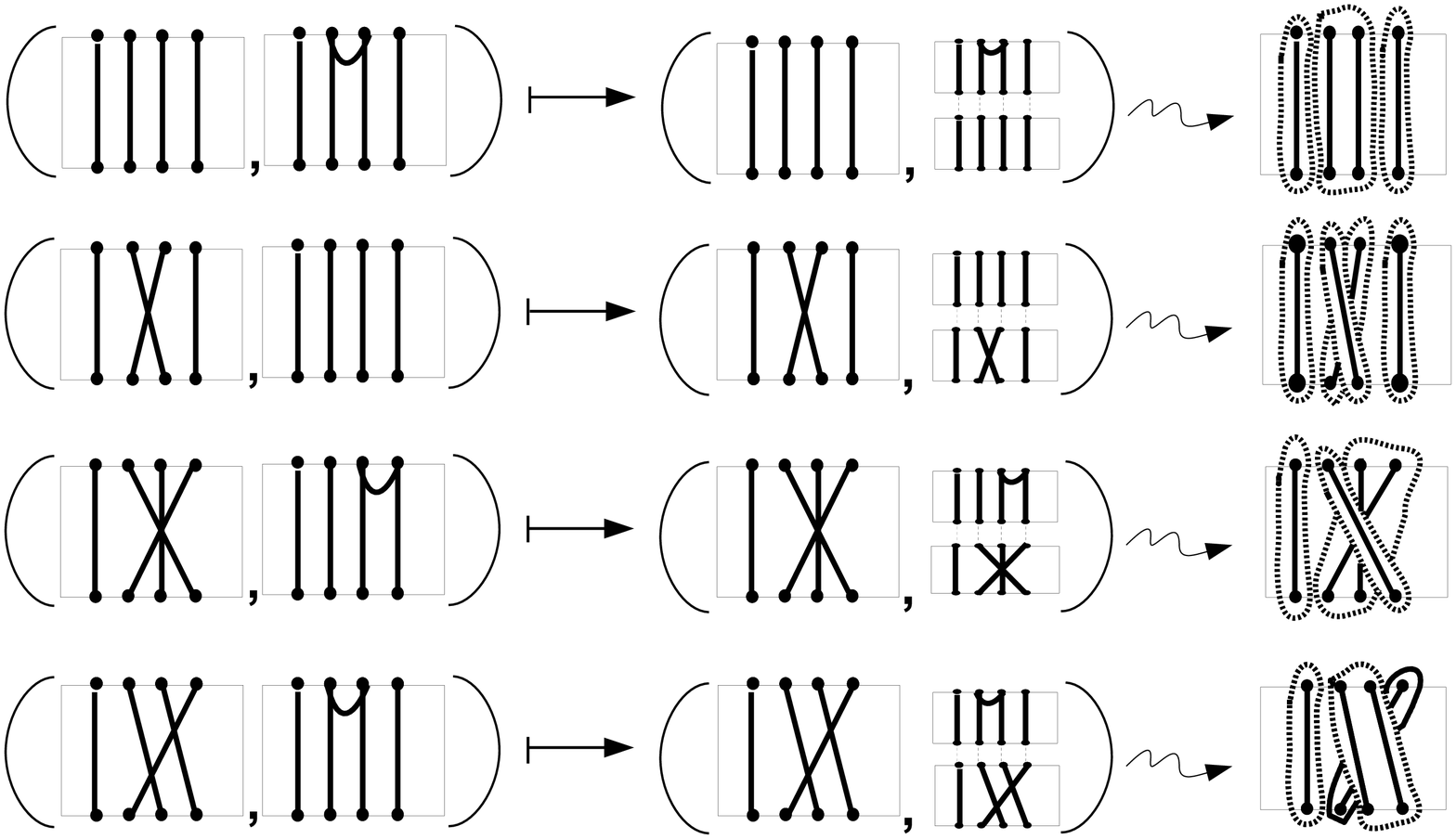}

\noindent In particular, these pictures describe the diagrammatic realization of some basis elements in $P_4^{\ltimes}.$  
\end{example}

\begin{coro}[See {\citep[\S 3.4]{martin}}]  \label{coro:dimSRP}
The dimension of $P_n^{\ltimes}$ is given by $n!B_n,$ where $B_n$ is the Bell number. \qed
\end{coro}

\begin{remark}
Notice that, $P_n^{\ltimes}$ is spanned by diagrams with $n$ propagating lines (See Example \ref{explRamified}). This means that, unlike the ramified partition algebras, the small ramified partition algebras do not depend on parameter $\delta.$     
\end{remark}

\begin{defn}
For any $\delta' \in F,$ we define $\Gamma_n$ as the subalgebra of $P_n(\delta')$ generated by $$\langle A^{i,j}  (i,j = 1,2, \ldots, n) \rangle .$$
\end{defn}

Note that the natural injection of $\Gamma_n$ into $P_n^{\ltimes}$ is given by $$A^{i,j} \mapsto (1,A^{i,j})$$ and there exists a natural injection of $FS_n$ into $P_n^{\ltimes}$ given by $$\sigma_{i,i+1} \mapsto (\sigma_{i,i+1},\sigma_{i,i+1}).$$ 

\begin{prop}[See {\citep[Prop. 2]{martin}}] \label{genSmll}
The algebra $P_n^{\ltimes}$ is generated by $(1,A^{i,i+1})$ and $(\sigma_{i,i+1}, \sigma_{i,i+1})$ ($i = 1,2, \ldots, n-1$). \qed
\end{prop}

%
\section{The Juyumaya algebra of braids and ties} \label{braidTie}
%

Following \cite{ryom2010}, we recall the Juyumaya algebra over the ring $\mathbb{C}[u,u^{-1}].$ 
  
\begin{defn}[See {\citep[\S 2]{ryom2010}}]
Let $u$ be an indeterminate over $\mathbb{C}$ and $\mathcal{A}$ be the principal ideal domain $\mathbb{C}[u,u^{-1}]$. The algebra $\mathcal{E}_n^{\mathcal{A}}(u)$ over $\mathcal{A}$ is the unital associative $\mathcal{A}$-algebra generated by the elements $T_1,~T_2, \ldots, T_{n-1}$ and $E_1,~E_2, \ldots, E_{n-1},$ which satisfy the defining relations

  \begin{eqnarray*}
    (A1) &T_iT_j = T_jT_i & \hbox{ if } ~|i-j| > 1\\
    (A2) &E_iE_j = E_jE_i & \forall \; i,j\\
    (A3) &E_i^2 = E_i \\
    (A4) &E_iT_i = T_iE_i\\
    (A5) &E_iT_j = T_jE_i & \hbox{ if } ~|i-j|>1\\
    (A6) &T_iT_jT_i = T_jT_iT_j & \hbox{ if } ~|i-j| = 1\\
    (A7) &E_jT_iT_j = T_iT_jE_i & \hbox{ if } ~|i-j| = 1\\
    (A8) &E_iE_jT_j = E_iT_jE_i = T_jE_iE_j & \hbox{ if } ~|i-j| = 1\\
    (A9) &T_i^2 = 1+(u-1)E_i(1-T_i) 
  \end{eqnarray*}
\end{defn}

Let $\mathbb{C}(u)$ be the field of rational function. We define $\mathcal{E}^0_n(u)$ as $$\mathcal{E}^0_n(u) := \mathcal{E}_n^{\mathcal{A}}(u) \otimes_{\mathcal{A}} \mathbb{C}(u)$$ where $\mathbb{C}(u)$ is made into an ${\mathcal{A}}$-module through inclusion.

\begin{coro}[See {\citep[Corollary 3]{ryom2010}}]  \label{coro:dimAJ}
The dimension of $\mathcal{E}^0_n(u)$ is given by $n!B_n,$ where $B_n$ is the Bell number. \qed
\end{coro}

The Bell number making appearance in Corollary \ref{coro:dimAJ} indicates that there might be a connection between the Juyumaya algebra and the (small ramified) partition algebra. In section \ref{mainRslt} we present this connection. 

From the presentation of $\mathcal{E}_n^{\mathcal{A}}(u),$ relations $(A1)$, $(A6)$, $(A9)$ form a deformation of the Coxeter relations (see \citep[\S 1]{mathas}) of the symmetric group $S_n$. It is straightforward to verify the following result.

\begin{prop}
There exists a homomorphism from $\mathcal{E}_n^{\mathcal{A}}(u)$ to the group ring $\mathbb{\mathcal{A}}S_n$ of the symmetric group given by 
\begin{eqnarray*}
X \colon \mathcal{E}_n^{\mathcal{A}}(u) & \to & \mathcal{A}S_n \\
T_i & \mapsto & \sigma_{i,i+1} \\
E_i & \mapsto & 0. 
\end{eqnarray*} \qed
\end{prop}

\noindent In particular, $\mathcal{A}S_n$ is isomorphic to a quotient of $\mathcal{E}_n^{\mathcal{A}}(u)$ when $u = 1.$

%
\section{Relationship of the Juyumaya algebra to $P_n^{\ltimes}$} \label{mainRslt} 
%

In response to a remark of Ryom-Hansen in \cite{ryom2010}, we present new results that establish a connection between the Juyumaya algebra and the partition algebra, via the small ramified partition algebra. 

Let $\mathbb{C}$ be the field of complex numbers which is a $\mathbb{C}[u,u^{-1}]$-algebra (that is, with $u$ specified to a complex number $x$). Denote the $\mathbb{C}$-algebra $\mathcal{E}_n^{\mathcal{A}}(u) \otimes_{\mathcal{A}} \mathbb{C}$ by $\mathcal{E}_n(x).$ Here, we shall only need the case $x = 1.$ 

\begin{prop} \label{prop:ajAlgHomo}
The map $\rho \colon \mathcal{E}_n(1) \to \mathbb{C}S_n \otimes_{\mathbb{C}} \Delta_n$ given by 
\begin{eqnarray*}
E_i & \mapsto & (1, A^{i,i+1})\\
T_i & \mapsto & (\sigma_{i,i+1}, \sigma_{i,i+1})
\end{eqnarray*}
defines a $\mathbb{C}$-algebra homomorphism.
\end{prop}

\begin{proof}
First, it is easy to see that the tensor product $\mathbb{C}S_n \otimes_{\mathbb{C}} \Delta_n$ makes sense. To show that this map is an algebra homomorphism we check that the relations (A1)--(A9) hold when $(1, A^{i,i+1})$ is put in place of $E_i$ and $(\sigma_{i,i+1}, \sigma_{i,i+1})$ is put in place of $T_i$ as follows.

\begin{itemize}
\item[(A1)] $\rho(T_iT_j) = (\sigma_{i,i+1}, \sigma_{i,i+1})\; (\sigma_{j,j+1},\sigma_{j,j+1})$ $= (\sigma_{i,i+1}\sigma_{j,j+1}, \; \sigma_{i,i+1}\sigma_{j,j+1})$ and \\
$\rho(T_jT_i) = (\sigma_{j,j+1}, \sigma_{j,j+1})$ $(\sigma_{i,i+1},\sigma_{i,i+1}) = (\sigma_{j,j+1}\sigma_{i,i+1},$ $\sigma_{j,j+1}\sigma_{i,i+1}).$\\
Since $|i-j|>1,$ $\sigma_{i,i+1}\sigma_{j,j+1} = \sigma_{j,j+1}\sigma_{i,i+1}.$ Thus, \\ 
$(\sigma_{i,i+1}\sigma_{j,j+1},$ $\sigma_{i,i+1}\sigma_{j,j+1}) = (\sigma_{j,j+1}\sigma_{i,i+1},$ $\sigma_{j,j+1}\sigma_{i,i+1})$ as required.

\bigskip
Diagrammatically, this may be represented as follows.

\includegraphics[scale=0.4]{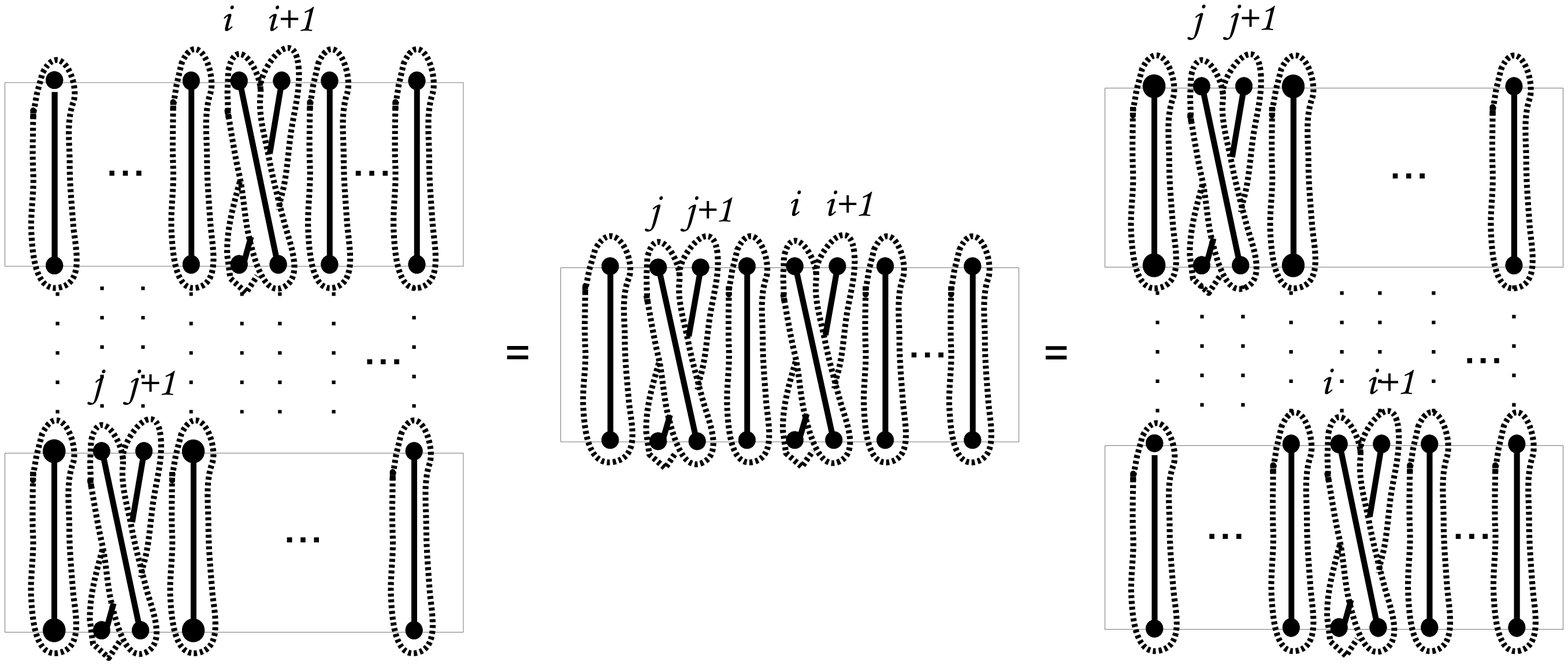}

\item[(A2)] $\rho(E_iE_j) = (1, A^{i,i+1})\; (1, A^{j,j+1}) = (1, \; A^{i,i+1}A^{j,j+1}) = (1, \; A^{j,j+1}A^{i,i+1})$\\
$ = (1, A^{j,j+1})\; (1, A^{i,i+1}) = \rho(E_jE_i)$.

The second equality follows from the definition of the tensor product of algebras, the third equality is a consequence of commutativity of $A^{k,k+1}$s and the fourth equality follows again from the definition of the tensor product of algebras.

\item[(A3)] $\rho(E_i^2) = (1, A^{i,i+1})\; (1, A^{i,i+1}) = (1, \; A^{i,i+1}A^{i,i+1}) = (1, \; A^{i,i+1}) = \rho(E_i).$
\item[(A4)] Similar to the proof of relation (A2), $\rho(E_iT_i) = (1, A^{i,i+1})\; (\sigma_{i,i+1}, \; \sigma_{i,i+1})$ \\ $= (1 \sigma_{i,i+1}, \; A^{i,i+1} \sigma_{i,i+1}) = (\sigma_{i,i+1} 1, \; \sigma_{i,i+1} A^{i,i+1})$\\ $= (\sigma_{i,i+1}, \;\sigma_{i,i+1}) \; (1, A^{i,i+1}) = \rho(T_iE_i).$
\item[(A5)] $\rho(E_iT_j) = (1, A^{i,i+1})\; (\sigma_{j,j+1}, \; \sigma_{j,j+1}) = (1 \sigma_{j,j+1}, \; A^{i,i+1} \sigma_{j,j+1})$ and \\ 
$\rho(T_jE_i) = (\sigma_{j,j+1}, \;\sigma_{j,j+1}) \; (1, A^{i,i+1}) = (\sigma_{j,j+1} 1, \; \sigma_{j,j+1} A^{i,i+1}).$\\
Since $|i-j|>1, \; A^{i,i+1} \sigma_{j,j+1} = \sigma_{j,j+1} A^{i,i+1}.$ Thus, \\
$(\sigma_{j,j+1}, \; A^{i,i+1} \sigma_{j,j+1}) = (\sigma_{j,j+1}, \; \sigma_{j,j+1} A^{i,i+1})$ as required.
\item[(A6)] $T_iT_jT_i$ corresponds to $(\sigma_{i,i+1}, \sigma_{i,i+1})\; (\sigma_{j,j+1}, \; \sigma_{j,j+1}) \; (\sigma_{i,i+1}, \sigma_{i,i+1}) = (\sigma_{i,i+1}\sigma_{j,j+1}\sigma_{i,i+1}, \sigma_{i,i+1}\sigma_{j,j+1}\sigma_{i,i+1})$ and \\ 
$T_jT_iT_j$ corresponds to $(\sigma_{j,j+1}, \sigma_{j,j+1})\; (\sigma_{i,i+1}, \; \sigma_{i,i+1}) \; (\sigma_{j,j+1}, \sigma_{j,j+1}) = (\sigma_{j,j+1}\sigma_{i,i+1}\sigma_{j,j+1}, \sigma_{j,j+1}\sigma_{i,i+1}\sigma_{j,j+1}).$\\
Since $|i-j|=1, \; \sigma_{i,i+1}\sigma_{j,j+1}\sigma_{i,i+1} = \sigma_{j,j+1}\sigma_{i,i+1}\sigma_{j,j+1}.$\\
Thus, $(\sigma_{i,i+1}\sigma_{j,j+1}\sigma_{i,i+1}, \sigma_{i,i+1}\sigma_{j,j+1}\sigma_{i,i+1}) = \\ (\sigma_{j,j+1}\sigma_{i,i+1}\sigma_{j,j+1}, \sigma_{j,j+1}\sigma_{i,i+1}\sigma_{j,j+1})$ as required.
\item[(A7)] The element $E_jT_iT_j$ is mapped to $(1, A^{j,j+1}) (\sigma_{i,i+1},\sigma_{i,i+1}) (\sigma_{j,j+1},\sigma_{j,j+1})$ \\
$ = (1 \sigma_{i,i+1}\sigma_{j,j+1}, \; A^{j,j+1}\sigma_{i,i+1}\sigma_{j,j+1})$ and \\
the element $T_iT_jE_i$ is mapped to $(\sigma_{i,i+1},\sigma_{i,i+1})  (\sigma_{j,j+1},\sigma_{j,j+1}) (1, A^{i,i+1}) \\ = (\sigma_{i,i+1}\sigma_{j,j+1} 1, \; \sigma_{i,i+1}\sigma_{j,j+1} A^{i,i+1}).$\\
Since $|i-j|=1, \; A^{j,j+1}\sigma_{i,i+1}\sigma_{j,j+1}= \sigma_{i,i+1}\sigma_{j,j+1}A^{i,i+1}$ and the result follows.

Proving relation (A7) using diagrams:

\includegraphics[scale=0.4]{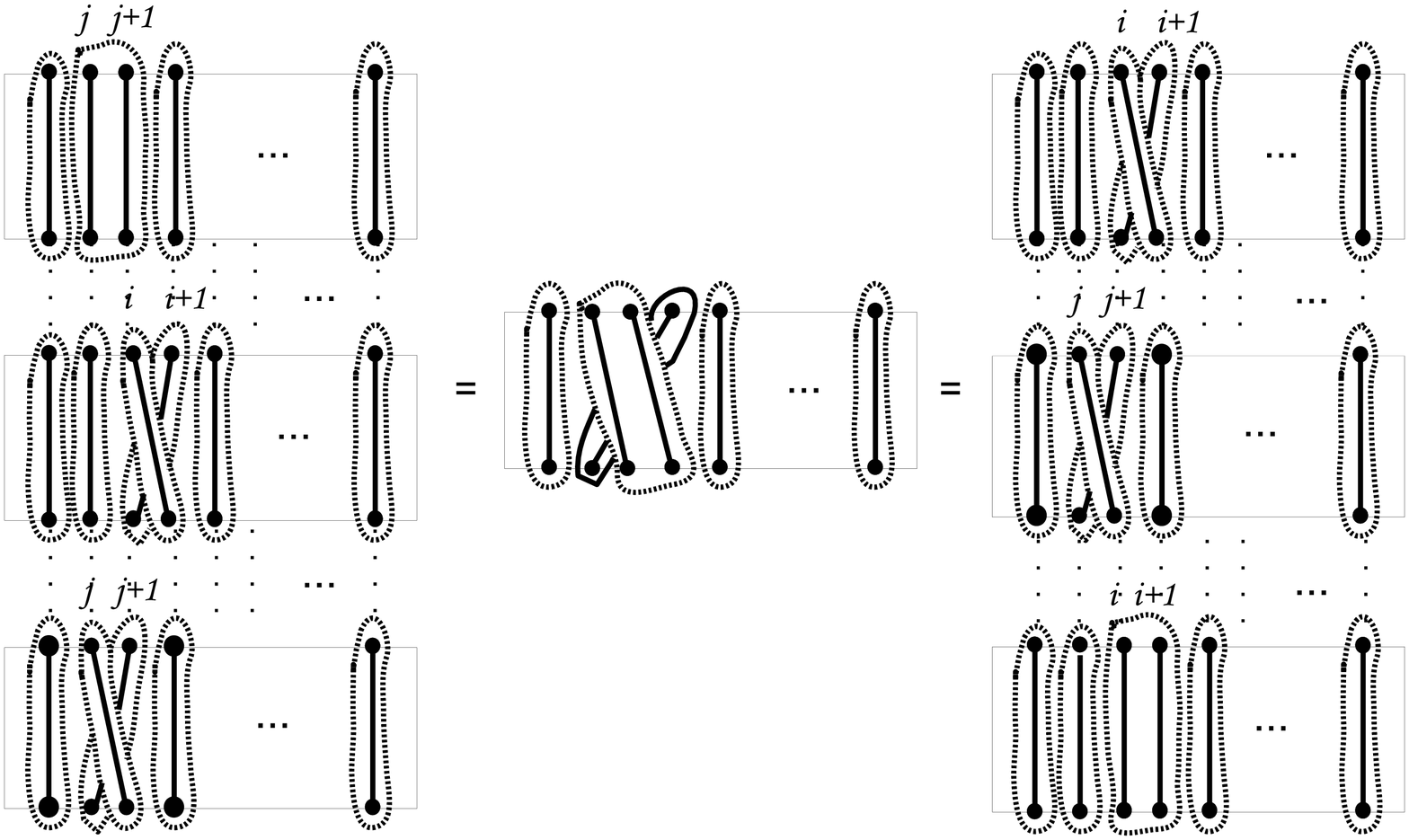}

\item[(A8)]The relation $E_iE_jT_j$ corresponds to $(1,A^{i,i+1})(1,A^{j,j+1})(\sigma_{j,j+1},\sigma_{j,j+1})$ $= (11\sigma_{j,j+1}, A^{i,i+1}A^{j,j+1}\sigma_{j,j+1}),$\\
the relation $E_iT_jE_i$ corresponds to $(1,A^{i,i+1})(\sigma_{j,j+1},\sigma_{j,j+1})(1,A^{i,i+1})= (1\sigma_{j,j+1}1, A^{i,i+1}\sigma_{j,j+1}A^{i,i+1}),$ and \\
the relation $T_jE_iE_j$ corresponds to $(\sigma_{j,j+1},\sigma_{j,j+1})(1,A^{i,i+1})(1,A^{j,j+1})= (\sigma_{j,j+1}11, \sigma_{j,j+1}A^{i,i+1}A^{j,j+1}).$\\
$A^{i,i+1}A^{j,j+1}\sigma_{j,j+1} = A^{i,i+1}\sigma_{j,j+1}A^{i,i+1} = \sigma_{j,j+1}A^{i,i+1}A^{j,j+1}$ since $|i-j| = 1$ as required.

Relation (A8) may be described using diagrams as follows.

\includegraphics[scale=0.4]{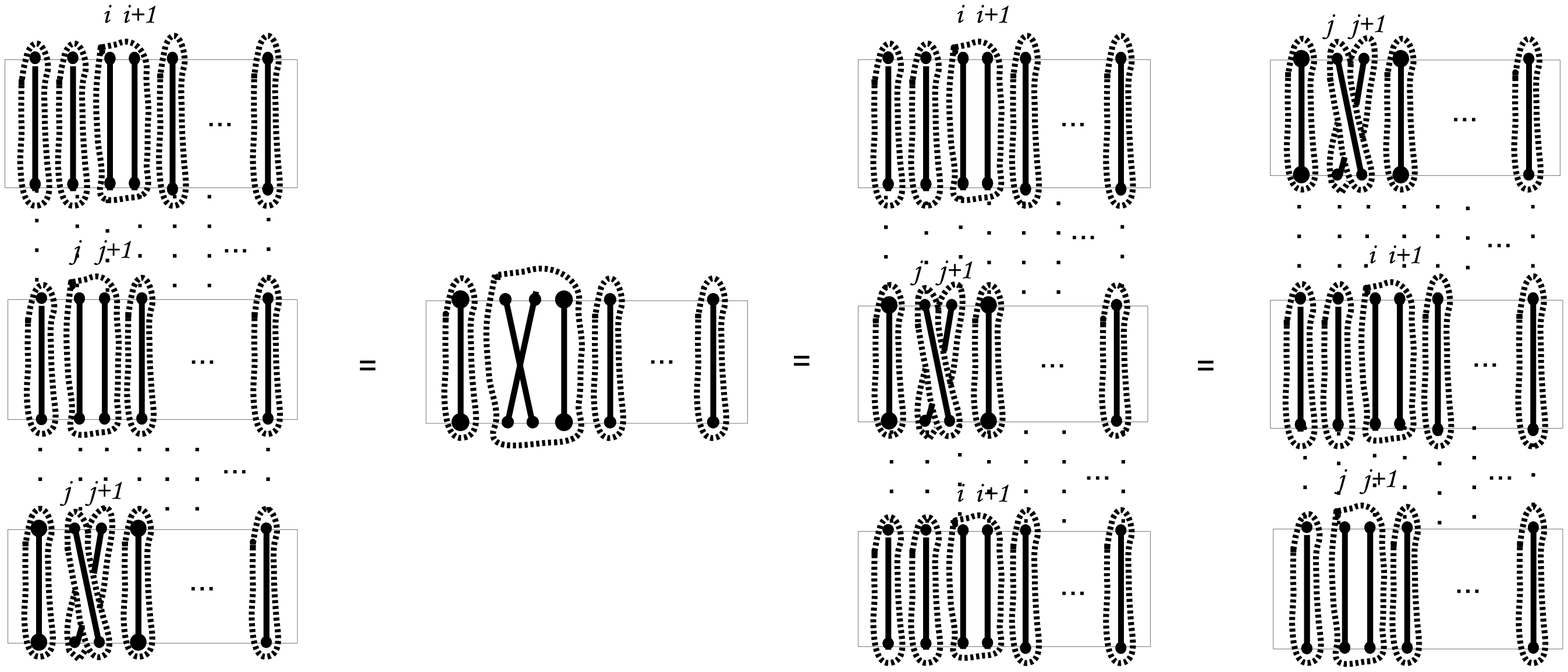}

\item[(A9)] Since $u$ is fixed to $1,$ it follows that $T_i^2 = 1$ and the relation corresponds to $(\sigma_{i,i+1}, \sigma_{i,i+1}) \; (\sigma_{i,i+1}, \sigma_{i,i+1}) = (\sigma_{i,i+1}\sigma_{i,i+1}, \sigma_{i,i+1}\sigma_{i,i+1}) = (1,1)$ as required.

\end{itemize}
\end{proof}

We leave it as an exercise to use diagrams to check relations (A2) -- (A6) and (A9). Next we show that

\begin{theo} \label{isoMap}
The map $\phi \colon \mathcal{E}_n(1) \to P_n^{\ltimes}$ given by 
\begin{eqnarray*}
E_i & \mapsto & (1, A^{i,i+1})\\
T_i & \mapsto & (\sigma_{i,i+1}, \sigma_{i,i+1})
\end{eqnarray*}
defines a $\mathbb{C}$-algebra isomorphism.
\end{theo}

\begin{proof}
The map $\phi$ is well-defined since by Proposition \ref{genSmll} $(1, A^{i,i+1})$ and $(\sigma_{i,i+1}, \sigma_{i,i+1})$ generates precisely $P_n^{\ltimes}$.

In order to check that $\phi$ is an algebra homomorphism, we need to verify that the defining relations of $\mathcal{E}_n(1)$ are satisfied in $P_n^{\ltimes}$ and this has already been shown in Proposition \ref{prop:ajAlgHomo}. All that remains is to show that the map is an isomorphism.
By Corollary \ref{coro:dimAJ} and by Corollary \ref{coro:dimSRP}, the dimensions of $\mathcal{E}_n(1)$ and $P_n^{\ltimes}$ are equal. Moreover, the map $\phi$ is surjective since the images of the generators generate $P_n^{\ltimes}$. 

Thus, the preceding facts together imply that $\phi$ is an isomorphism.
\end{proof}

%
\section{Representation theory} \label{repthry}
%

Generic irreducible representations of the Juyumaya algebra are constructed for the cases $n = 2,3$ in \cite{aicardi}, \cite{ryom2010}. Here we do all other cases, by reference to \cite{martin}.

In the previous section we established, for each $n \in \mathbb{N},$ an isomorphism between the algebras $\mathcal{E}_n(1)$ and $P_n^{\ltimes}.$ With this result we have implicitly determined the complex representation theory of $\mathcal{E}_n(1)$ since the representation theory of $P_n^{\ltimes}$ over $\mathbb{C}$ is already known (cf. \cite{martin}) and    

\begin{theo}[See {\citep[Theorem 4]{martin}}] \label{srpSplit}
For all $n,$ the algebra $P_n^{\ltimes}$ over $\mathbb{C}$ is split semisimple. \qed
\end{theo}

We can now prove that

\begin{theo}
For all $n,$ the algebra $\mathcal{E}_n(x)$ is generically semisimple.
\end{theo}
(By \emph{generically} we mean in a Zariski open subset of the complex space)

\begin{proof}
By Theorem \ref{isoMap}, $\mathcal{E}_n(1)$ is isomorphic to the algebra $P_n^{\ltimes}$ and by Theorem \ref{srpSplit} $P_n^{\ltimes}$ is split semisimple over $\mathbb{C}$. This implies that $\mathcal{E}_n(1)$ is split semisimple over $\mathbb{C}$. Split semisimplicity is a generic property according to Parshall, Cline and Scott \citep[\S 1]{cline1999}. That is, the split semisimplicity property holds (in our case) on a Zariski open subset (see \cite{smith2000}) of complex space. Therefore, $\mathcal{E}_n(x)$ is split semisimple for generic choices of $x \in \mathbb{C}.$ But $\mathcal{E}_n(x)$ is semisimple if and only if it is split semisimple since we are working over an algebraically closed field of characteristic zero.
\end{proof}

\noindent {\bf Acknowledgements.} The results of this paper form part of the author's doctoral thesis. The author would like to express her gratitude to Paul Martin for introducing her to this field and for helpful discussions and encouragement, and to Andrew Reeves, Alison Parker, Maud De Visscher, for useful comments on the earlier version of the paper. She would also like to thank City University London (research studentship) for financial support. She is grateful to the University of Leeds for hospitality.

\bibliographystyle{plain}
\bibliography{sample1}

\end{document}